\documentclass[a4paper,article]{memoir}
\usepackage{amsthm,amsmath,amssymb}
\newcommand{\D}{\textup{d}}
\newcommand{\dvol}{\textup{dvol}}
\newcommand{\Vol}{\textup{Vol}}

\newcommand{\Div}{\textup{div}}
\newcommand{\tr}{\textup{tr}}
\newcommand{\eps}{\varepsilon}
\newtheorem{thm}{Theorem}[chapter]
\newtheorem{prop}[thm]{Proposition}
\newtheorem{cor}[thm]{Corollary}
\theoremstyle{definition}
\newtheorem{ex}[thm]{Example}
\setlrmarginsandblock{3.5cm}{3.5cm}{*}
\checkandfixthelayout
\title{Monotonicity for $p$-harmonic vector bundle-valued $k$-forms}
\author{Ahmad Afuni}
\date{\vskip -5ex}
\begin{document}
\maketitle
\begin{abstract}We investigate monotonicity properties of $p$-harmonic vector bundle-valued $k$-forms by studying the energy-momentum tensor associated with such a form. As a consequence, we obtain a unified proof of the monotonicity formul\ae\ for $p$-harmonic maps and Yang-Mills connections, proving a monotonicity formula for $p$-Yang-Mills connections in the process. Moreover, it is shown how this technique may be adapted to yield an analogous monotonicity formula for Yang-Mills-Higgs pairs. Finally, we obtain Liouville-type theorems for such forms and Yang-Mills-Higgs pairs as an application.\end{abstract}
\chapter{Introduction}
Let $(M,g)$ be an oriented Riemannian manifold of dimension $n>kp$ with $p>1$ and $k\in\mathbb{N}$ fixed, $\Lambda^{k} T^{\ast}M$ the $k$th exterior product bundle of the cotangent bundle of $M$ and $E\rightarrow M$ a finite-dimensional Riemannian vector bundle with connection $\nabla$, exterior covariant differential $\D^{\nabla}$ and codifferential $\delta^{\nabla}$ (see \S\ref{geomsec}).

A smooth section $\psi:M\rightarrow E\otimes\Lambda^{k}T^{\ast}M$ is said to be \textit{$p$-harmonic} if it is $\D^{\nabla}$-closed, i.e.
\begin{align}
       	\D^{\nabla}\psi=0\label{closed}
\end{align}
and $p$-coclosed, i.e.
\begin{align}
\delta^{\nabla}(|\psi|^{p-2}\psi)=0.\label{coclosed}
\end{align}
These equations have been studied by countless others in the case where $E=M\times\mathbb{R}\rightarrow M$ and $\psi=d^{\nabla}v$ (cf. \cite{MR1652927} and the references therein). Besides considering them for their own sake, geometric variational problems such as $p$-harmonic maps and Yang-Mills theory may be recast in this form (see \S\ref{geomsec}); for the former, monotonicity formul\ae\ have been established by Schoen and Uhlenbeck \cite{MR664498} and Hardt and Lin \cite{CPA:CPA3160400503}, and for the latter by Price \cite{price1983monotonicity}. These formul\ae\ are special cases of the following theorem, which was established by Karcher and Wood \cite{MR765831} in the case $p=2$.

\begin{thm} Let $\psi:M\rightarrow E\otimes \Lambda^k T^\ast M$ be a $p$-harmonic section, $x_{0}\in M$, $i_{0}$ the injectivity radius at $x_{0}$ and $\dvol_{g}$ the volume form of $(M,g)$. There exists a constant $\Lambda\geq 0$ depending on the geometry of $B_{i_{0}}(x_0)$ such that the identity
\begin{align}
	\frac{\D}{\D R}\left( \frac{e^{\Lambda R^{2}}}{R^{n-kp}}\int_{B_R (x_0)}\frac{1}{p}|\psi|^{p}\dvol_g\right)\geq 0  
	\end{align}
	holds on $\left]0,\frac{i_{0}}{2}\right[$.
\end{thm}

The purpose of this note is to establish this theorem by exploiting a divergence identity arising from the so-called \textit{energy-momentum tensor} associated to the integrand, thus reproving the now well-known monotonicity formul\ae\ mentioned earlier as well as proving new ones for $p$-Yang-Mills connections and, by suitably modifying our setup, Yang-Mills-Higgs pairs. This approach was motivated by a paper of Alikakos \cite{alikakos2011some} where an energy-momentum tensor was used to establish a monotonicity formula for a certain semilinear elliptic system in $\mathbb{R}^{n}$. The tensor itself, however, is of independent interest, playing a major role in the theory of relativity \cite{MR0143451} and having been studied in the context of harmonic maps by Eells and Baird \cite{MR655417} and various others. In a forthcoming paper, it shall be shown how this identity may be used to establish local monotonicity formul\ae\ for related geometric flows.
\vskip 2mm
\noindent{\small\textit{Acknowledgements.} This research was mostly carried out as part of the author's doctoral thesis at the Free University of Berlin under the supervision of Klaus Ecker, to whom much gratitude is due. The author gratefully acknowledges financial support from the Max Planck Institute for Gravitational Physics and the Leibniz Universit\"at Hannover.}
\chapter{Geometric setup and problems of note}\label{geomsec}
We begin by giving the geometric setup underlying this paper, fixing notation in the process. As a rule, we follow the conventions of \cite{walter1981poor}.

As in the introduction, we shall assume that $(M^{n},g)$ is an oriented Riemannian manifold of dimension \textit{greater} than $kp$ ($p>1$ and $k\in\mathbb{N}$ fixed) with volume form $\dvol_{g}$ and furthermore write $TM$ for its tangent bundle, $T^{\ast}M$ for its cotangent bundle, $\Lambda^{k}T^{\ast}M$ for the $k$th exterior product bundle of $T^{\ast}M$ and $(\Lambda T^{\ast}M,\wedge)$ for the exterior algebra bundle of $T^{\ast}M$ with wedge product $\wedge$; all of these bundles naturally admit Riemannian metrics induced by that on $TM$. Moreover, we suppose $E\rightarrow M$ is a finite-dimensional Riemannian vector bundle equipped with a connection $\nabla$ and write $\left<\cdot,\cdot\right>$ for the Riemannian metric on $E$ and, more generally, for the metrics canonically induced on bundles `naturally' constructed from $E$ and $TM$, writing $|\cdot|$ for the associated norm in all cases and $\nabla$ for any connection naturally induced by the connection on $E$ and the Levi-Civita connection on $TM$. With these conventions, $\nabla$ is compatible with all of the inner products $\left<\cdot,\cdot\right>$ to be considered in this paper. If $E_{0}\rightarrow M$ is any vector bundle, we write $\Gamma(E_{0})$ for the $C^{\infty}(M)$-module of all smooth sections of $E_{0}$. Throughout this paper, we work in the \textit{smooth} category.

Associated to $\nabla$ is the so-called \textit{exterior covariant derivative} $\D^{\nabla}:\Gamma(E\otimes \Lambda T^{\ast}M)\rightarrow\Gamma(E\otimes\Lambda T^{\ast}M)$ given by
\begin{align*}
	\D^{\nabla} = \sum_{i=1}^{n}\omega^{i}\wedge \nabla_{\eps_{i}}
\end{align*}
in any local frame $\{\eps_{i}\}_{i=1}^{n}$ for $TM$ with dual coframe $\{\omega^{i}\}_{i=1}^{n}$ for $T^{\ast}M$. Moreover, writing $\iota_{X}:\Gamma(E\otimes\Lambda T^{\ast}M)\rightarrow\Gamma(E\otimes\Lambda T^{\ast}M)$ for the interior product associated to a vector field $X\in\Gamma(TM)$, we define the associated \textit{codifferential} $\delta^{\nabla}:\Gamma(E\otimes\Lambda T^{\ast}M)\rightarrow \Gamma(E\otimes \Lambda T^{\ast}M)$ by
\begin{align*}
	\delta^{\nabla} = -\sum_{i=1}^{n}\iota_{\eps_{i}}\circ\nabla_{\eps_{i}}
\end{align*}
in any local \textit{orthonormal} frame $\{\eps_{i}\}_{i=1}^{n}$ for $TM$, which arises as the adjoint to $\D^{\nabla}$ with respect to the canonical $L^{2}$-inner product associated to $\left<\cdot,\cdot\right>$ acting on compactly supported sections of $E\otimes\Lambda T^{\ast}M$. Explicitly,
\begin{align*}
	\int_{M}\left<\D^{\nabla}\psi_{1},\psi_{2}\right>\dvol_{g}=\int_{M}\left<\psi_{1},\delta^{\nabla}\psi_{2}\right>\dvol_{g}
\end{align*}
whenever $\psi_{1},\psi_{2}\in\Gamma(E\otimes\Lambda T^{\ast}M)$ are compactly supported. Since all of these operators are local, we shall freely apply them to local sections.

We now proceed to mention a few examples of systems that may be written in the form (\ref{closed})-(\ref{coclosed}).
\begin{ex}[$k\in\mathbb{N}$: $p$-harmonic forms] If $E=M\times\mathbb{R}\rightarrow M$ with inner product given by fibrewise multiplication and equipped with the usual flat connection then, with the identification $E\otimes\Lambda T^{\ast}M\cong \Lambda T^{\ast}M$, $\D^{\nabla}$ and $\delta^{\nabla}$ reduce to the usual exterior differential and codifferential of Hodge theory and the equations (\ref{closed})-(\ref{coclosed}) define a \textit{$p$-harmonic $k$-form} $\psi\in\Gamma(\Lambda^{k}T^{\ast}M)$; such forms arise as solutions to the variational problem
\begin{align*}
	\frac{1}{p}\int_{M}|\omega|^{p}\dvol_{g}\rightarrow\textup{min!}
\end{align*}
considered over the class of all \textit{closed} $k$-forms $\omega$ on which this integral is finite.
\end{ex}
\begin{ex}[$k=1$: $p$-harmonic maps {\cite{fuchs89}}]\label{hmex} Let $u:M\rightarrow (N,g_{N})$ be a smooth mapping of Riemannian manifolds and $E=u^{-1}TN$ the pullback of $TN$ by $u$. The Riemannian metric and Levi-Civita connection on $TN$ induce a Riemannian metric $\left<\cdot,\cdot\right>$ and connection $\nabla$ respectively on $u^{-1}TN$, the latter of which is compatible with the Riemannian metric. Write $\psi=\D u$ for the differential of $u$, here considered a section of $u^{-1}TN\otimes T^{\ast}M$. The condition $\D^{\nabla}\D u=0$ holds for all $u$ as a consequence of the connection on $TN$ being torsion-free. On the other hand, if $\psi$ satisfies (\ref{coclosed}), $u$ is said to be a \textit{$p$-harmonic map} (simply a \textit{harmonic map} when $p=2$). These maps arise as solutions to the variational problem
\begin{align*}
	\frac{1}{p}\int_{M}|\D v|^{p}\dvol_{g}\rightarrow\textup{min!}
\end{align*}
considered over the class of all smooth maps $v:M\rightarrow N$ on which this integral is finite.
\end{ex}
\begin{ex}[$k=2$: $p$-Yang-Mills connections {\cite{rivkes08}}]\label{ymex} Suppose $G\rightarrow P\rightarrow M$ is a principal fibre bundle with compact connected semi-simple structure group $G$ with Lie algebra $\mathfrak{g}$ and write $E$ for the vector bundle associated to $P$ and the adjoint representation of $G$ on $\mathfrak{g}$. Minus the Killing form induces a Riemannian metric $\left<\cdot,\cdot\right>$ on $E$; moreover, given a connection $\omega$ on $P$, realised here as a $\mathfrak{g}$-valued one-form on $P$, there is a natural associated connection $\nabla$ on $E$ which is compatible with $\left<\cdot,\cdot\right>$. Writing $\psi=\underline{\Omega}^{\omega}\in\Gamma(E\otimes \Lambda^{2}T^{\ast}M)$ for the \textit{curvature two-form} associated to $\omega$, we say that $\omega$ is a \textit{$p$-Yang-Mills connection} (simply a \textit{Yang-Mills connection} when $p=2$) if $\psi$ solves (\ref{coclosed}). Similarly to the preceding example, $\psi$ satisfies (\ref{closed}) for all $\omega$, the statement of which is known as the \textit{Bianchi identity}. Such connections arise as solutions to the variational problem
\begin{align*}
	\frac{1}{p}\int_{M}|\underline{\Omega}^{\xi}|^{p}\dvol_{g}\rightarrow\textup{min!}
\end{align*}
considered over the class of all connections $\xi$ on $P$ on which this integral is finite.
\end{ex}
Common to all of the above examples is an energy of the form $\frac{1}{p}\int_{M}|\psi|^{p}\dvol_{g}$, thus suggesting that a closer study of the integrand might be fruitful. We henceforth assume that $p>1$ is fixed and write $e_{g}(\psi)=\frac{1}{p}|\psi|^{p}$ for the energy density, explicitly indicating the metric $g$ and vector bundle-valued form $\psi\in\Gamma(E\otimes \Lambda^{k}T^{\ast}M)$ to be assumed given in all of the discussions to follow and, whenever necessary, explicitly indicating the dependence of the inner product on $E\otimes\Lambda T^{\ast}M$ on the metric $g$ by writing it as $\left<\cdot,\cdot\right>_{g}$.
\chapter{The energy-momentum tensor}\label{emsec}
We now proceed to investigate the dependence of the energy density $e_{g}(\psi)$ on the metric $g$. Throughout this section we make use of the canonical bilinear pairing $E\otimes\Lambda T^{\ast}M\times \Lambda TM \rightarrow E$, also known as `evaluating bundle-valued forms on vectors,' which we write simply as $(\cdot,\cdot)$. Moreover, we write $\eps_{J}=\eps_{j_{1}}\wedge\dots\wedge\eps_{j_{l}}=\eps_{j_{1}\dots j_{l}}$ whenever $\{\eps_{i}\}_{i=1}^{n}$ is a local frame for $TM$, $J=(j_{1},\dots,j_{l})$ an $l$-multi-index for $l\in\mathbb{N}$ and similarly for $\omega^{J}$ whenever $\{\omega^{i}\}_{i=1}^{n}$ is a local frame for $T^{\ast}M$. We always understand sums over multi-indices $J$ to be over \textit{increasing} multi-indices and write $J^{l}$ for $J$ whenever necessary.

We begin with a proposition stating how the energy density, considered as an $n$-form, varies with $g$.
\begin{prop}\label{emdef} The unique (symmetric) tensor  $T_{\psi}^{g}\in\Gamma(T^{\star}M\otimes T^{\star}M)$ satisfying
\begin{align*}
\left.\frac{\D}{\D t}\right|_{t=0}e_{g(t)}(\psi)(x)\dvol_{g(t)}(x)=\left<-\frac{1}{2}T_{\psi}^{g}(x),h(x)\right>_{g}\dvol_{g}(x)
\end{align*}
for all $x\in M$ whenever $\{g(t)\in\Gamma(T^{\ast}M\otimes T^{\ast}M)\}_{t\in\left]-\eps,\eps\right[}$ is a smooth one-parameter family of metrics with $g(0)_{p}=g_{p}$ and $\left.\frac{\D}{\D t}\right|_{t=0}g(t)_{p}=h(p)$ for all $p\in M$ is given in any local frame $\{\eps_{i}\}\leftrightarrow\{\omega^{i}\}$ by
\begin{align*}
T_{\psi}^{g}&=|\psi|^{p-2}\sum_{i,j=1}^{n}\left<\iota_{\eps_{i}}\psi,\iota_{\eps_{j}}\psi\right>\omega^{i}\otimes \omega^{j} - e_{g}(\psi)g.
\end{align*}
\end{prop}
\begin{proof}
	On the one hand, it is clear that
\begin{align*}
\left.\frac{\D}{\D t}\right|_{t=0}\dvol_{g(t)}(x)=\left(\frac{1}{2}\left<g,h\right>_{g}\dvol_{g}\right)(x)
\end{align*}
for all $x\in M$ for, computing in a co\"ordinate neighbourhood and writing $g_{ij}$ and $h_{ij}$ ($i,j\in\{1,\dots,n\}$) for the components of $g$ and $h$ so that $\dvol_{g}=\sqrt{\det(g_{ij})}\D x$, we have
\begin{align*}
	\left.\frac{\D}{\D t}\right|_{t=0}\dvol_{g(t)}(x)&=\left(\frac{1}{2\sqrt{\det(g_{ij})}}\sum\limits_{i,j}g^{ji}\det(g_{ij}) h_{ij}\D x\right)(x)\\
	&=\left(\frac{1}{2}\sqrt{\det(g_{ij})}\sum\limits_{i,j}g^{ij}h_{ij}\D x\right)(x)=\left(\frac{1}{2}\mbox{tr}_{g}h\dvol_{g}\right)(x).
\end{align*}
On the other hand, by the equality
\begin{align}
\left.\frac{\D}{\D t}\right|_{t=0}\left<\psi,\psi\right>_{g(t)}&=\left<-\sum_{i,j=1}^{n}\left(\left<\iota_{\eps_{i}}\psi,\iota_{\eps_{j}}\psi  \right>\right)\omega^{i}\otimes\omega^{j},h \right>_{g},\label{varip}
\end{align}
the result follows from
\begin{align*}
	\left.\frac{\D}{\D t}\right|_{t=0}e_{g(t)}(\psi)&=\frac{1}{2}|\psi|^{p-2}\left.\frac{\D}{\D t}\right|_{t=0}\left<\psi,\psi\right>_{g(t)}.
\end{align*}
To establish (\ref{varip}), we fix a local $g$-orthonormal frame $\{\eps_{i}\}$ for $TM$ with dual coframe $\{\omega^{i}\}$ for $T^{\ast}M$ and note that
\begin{align*}
\left<\psi,\psi\right>_{g(t)}&=\sum_{I^{k},J^{k}}\left<(\psi,\eps_{I}),(\psi,\eps_{J})\right>_{g(t)}\cdot\det\left((g^{\ast}(t),\omega^{i_{r}}\otimes \omega^{j_{s}})\right)_{r,s=1}^{k}
\end{align*}
where $g^{\ast}(t)$ is the metric on $T^{\ast}M$ induced by $g(t)$, and
\begin{align*}
&\left.\partial_{t}\right|_{t=0}\left(\det\left((g^{\ast}(t),\omega^{i_{r}}\otimes\omega^{j_{s}})\right)   \right)\\
&=-\sum_{r,s=1}^{k}(h,\eps_{i_{r}}\otimes\eps_{j_{s}})\cdot(-1)^{r+s}\left<\omega^{i_{1}\dots i_{r-1}\widehat{i_{r}}i_{r+1}\dots i_{k}},\omega^{j_{1}\dots j_{s-1}\widehat{j_{s}}j_{s+1}\dots j_{k}}\right>,
\end{align*}
where $\widehat{\cdot}$ denotes omission, whence, noting that $\omega^{i_{1}\dots i_{r-1}\widehat{i_{r}}i_{r+1}\dots i_{k}}=(-1)^{r+1}\iota_{\eps_{i_{r}}}\omega^{I}$,
\begin{align*}
&\left.\partial_{t}\right|_{t=0}\left<\psi,\psi\right>_{g(t)}\\
&=-\sum_{I^{k},J^{k}}\sum_{r,s=1}^{k}\left<(\psi,\eps_{I}),(\psi,\eps_{J})\right>\cdot (h,\eps_{i_{r}}\otimes\eps_{j_{s}})\cdot\left<\iota_{\eps_{i_{r}}}\omega^{I},\iota_{\eps_{j_{s}}}\omega^{J}\right>.
\end{align*}
The inner sum is invariant under permutations of $I$ and $J$, i.e. under $I\rightarrow\sigma(I)$, $J\rightarrow\tau(J)$ for any $\sigma,\tau:\{1,\dots,k\}\rightarrow\{1,\dots,k\}$ bijective. We proceed with this in mind:
\begin{align}
=-\frac{1}{(k!)^{2}}\sum_{\substack{i_{1},\dots,i_{k}\\j_{1},\dots,j_{k}}}\sum_{r,s=1}^{k}\left<(\psi,\eps_{I}),(\psi,\eps_{J})\right>\cdot (h,\eps_{i_{r}}\otimes\eps_{j_{s}})\cdot\left<\iota_{\eps_{i_{r}}}\omega^{I},\iota_{\eps_{j_{s}}}\omega^{J}\right>.\label{permthing}
\end{align}
Now, interchanging sums and fixing $r,s$, we note that the inner summand may be written as, writing $\sigma(I)=(i_{r},i_{1},\dots,\widehat{i_{r}},\dots,i_{k})$ and $\tau(J)=(j_{s},j_{1},\dots,\widehat{j_{s}},\dots,j_{k})$,
\begin{align*}
&\sum_{\substack{i_{1},\dots,i_{k}\\j_{1},\dots,j_{k}}}(-1)^{r+s}\left<(\psi,\eps_{\sigma(I)}),(\psi,\eps_{\tau(J)})\right>\cdot (h,\eps_{(\sigma(I))_{1}}\otimes\eps_{\tau(J)_{1}})\cdot(-1)^{r+s}\left<\iota_{\eps_{\sigma(I)_{1}}}\omega^{\sigma(I)},\iota_{\eps_{\tau(J)_{1}}}\omega^{\tau(J)}\right>\\
&=\sum_{\substack{i_{1},\dots,i_{k}\\j_{1},\dots,j_{k}}}\left<(\iota_{\eps_{i_{1}}}\psi,\eps_{i_{2}\dots i_{k}}),(\iota_{\eps_{j_{1}}}\psi,\eps_{i_{2}\dots i_{k}})\right>\cdot (h,\eps_{i_{1}}\otimes\eps_{j_{1}})\cdot\left<\iota_{\eps_{i_{1}}}\omega^{I},\iota_{\eps_{j_{1}}}\omega^{J}\right>,
\end{align*}
where we have made a change of variables. Noting now that this expression does not depend on $r$ and $s$ so that, summing over $r$ and $s$, we obtain $k^{2}$ of these sums and, treating $i_{1}$ and $j_{1}$ as separate variables from the other $i_{\cdot}$ and $j_{\cdot}$, we proceed from (\ref{permthing}), rewriting the outer sum in terms of increasing multi-indices, to obtain
\begin{align*}
&-\sum_{P^{k-1},Q^{k-1}}\sum_{i_{1},j_{1}=1}^{n}\left<(\iota_{\eps_{i_{1}}}\psi,\eps_{P}),(\iota_{\eps_{j_{1}}}\psi,\eps_{Q})\right>(h,\eps_{i_{1}}\otimes\eps_{j_{1}})\underbrace{\left<\omega^{P},\omega^{Q}\right>}_{\delta^{PQ}}\\
&=-\sum_{i,j=1}^{n}\left<\iota_{\eps_{i}}\psi,\iota_{\eps_{j}}\psi\right>\left(h,\eps_{i}\otimes\eps_{j}\right)\\
&=-\left<h,\sum_{i,j=1}^{n}\left<\iota_{\eps_{i}}\psi,\iota_{\eps_{j}}\psi\right>\omega^{i}\otimes\omega^{j}\right>,
\end{align*}
which is independent of the choice of frame.

\end{proof}
We call $T_{\psi}^{g}$ the \textit{energy-momentum tensor} associated to $e_{g}(\psi)$. In \cite{alikakos2011some}, Alikakos considered the system
\begin{align}
	\Delta u - \nabla W(u)=0\label{semilineqn}
\end{align}
for $u\in C^{2}\left(\mathbb{R}^{n},\mathbb{R}^{n}\right)$ and $W\in C^{2}(\mathbb{R}^{n},\mathbb{R}^{+})$, which is naturally associated to the energy
\begin{align*}
\int_{\mathbb{R}^{n}}\frac{1}{2}|\D u|^{2} + W\circ u;
\end{align*}
there, the energy-momentum tensor is
\begin{align*}
T_{ij}&=\partial_{i}u\cdot\partial_{j}u -\left(\frac{1}{2}|\D u|^{2} + W\circ u   \right)\delta_{ij},
\end{align*}
which was shown to enjoy the property $\Div\ T=0$ that ultimately led to a monotonicity formula. This suggests that computing the divergence of $T_{\psi}^{g}$ should lead to something useful.
\begin{prop}\label{emdiv} In any local frame $\{\eps_{i}\}\leftrightarrow\{\omega^{i}\}$ for $TM$ and $T^{\ast}M$,
\begin{align*}
\mbox{div}\ T_{\psi}^{g}&=-\sum_{j=1}^{n}\left(\left<\delta^{\nabla}(|\psi|^{p-2}\psi),\iota_{\eps_{j}}\psi\right> + \left<|\psi|^{p-2}\iota_{\eps_{j}}d^{\nabla}\psi,\psi\right>\right)\omega^{j}.
\end{align*}
\end{prop}
\begin{proof}
We compute in a local orthonormal frame adapted at $x\in M$, evaluating all of the following functions at $x$:
\begin{align*}
\mbox{div}\ T_{\psi}^{g}&=-\sum_{j}\left<\delta^{\nabla}(|\psi|^{p-2}\psi),\iota_{\eps_{j}}\psi \right>\omega^{j} \\
&\qquad +|\psi|^{p-2}\sum_{i,j}\sum_{J^{k-1}}\left(\left<(\psi,\eps_{i}\wedge\eps_{J}),(\nabla_{\eps_{i}}\psi,\eps_{j}\wedge\eps_{J}) \right> - \frac{1}{k}\left<(\nabla_{\eps_{j}}\psi,\eps_{i}\wedge\eps_{J}),(\psi,\eps_{i}\wedge\eps_{J})  \right>  \right)\omega^{j}.
\end{align*}
Now, we note that
\begin{multline*}
\left(\D^{\nabla}\psi,\eps_{jij_{1}\dots j_{k-1}}\right)\\
=(\nabla_{\eps_{j}}\psi,\eps_{ij_{1}\dots j_{k-1}})-(\nabla_{\eps_{i}}\psi,\eps_{jj_{1}\dots j_{k-1}}) +\sum_{q=1}^{k-1}(-1)^{q+1}(\nabla_{\eps_{j_{q}}}\psi,\eps_{jij_{1}\dots j_{q-1}\widehat{j}_{q}j_{q+1}\dots j_{k-1}}),
\end{multline*}
whence
\begin{align}
&\sum_{i}\sum_{J^{k-1}}\left<(\nabla_{\eps_{j}}\psi,\eps_{i}\wedge\eps_{J}),(\psi,\eps_{i}\wedge\eps_{J})  \right>\notag\\
&=\frac{1}{(k-1)!}\sum_{i,j_{1},...,j_{k-1}}\left<(\nabla_{\eps_{j}}\psi,\eps_{ij_{1}\dots j_{k-1}}),(\psi,\eps_{ij_{1}\dots j_{k-1}})\right>\notag\\
&=\sum_{i}\sum_{J^{k-1}}\left(\left<(d^{\nabla}\psi,\eps_{j}\wedge\eps_{i}\wedge\eps_{J}),(\psi,\eps_{i}\wedge\eps_{J})   \right> +\left<(\nabla_{\eps_{i}}\psi,\eps_{j}\wedge\eps_{J}),(\psi,\eps_{i}\wedge\eps_{J})\right>\right)\notag\\
& +\frac{1}{(k-1)!}\sum_{i,j_{1},\dots, j_{k-1}}\sum_{q=1}^{k-1}(-1)^{q}\left<(\nabla_{\eps_{j_{q}}}\psi,\eps_{jij_{1}\dots j_{q-1}\widehat{j}_{q}j_{q+1}\dots j_{k-1}}),(\psi,\eps_{ij_{1}\dots j_{k-1}})\right>.\label{simpemdiv}
\end{align}
Expanding the sum over $q$ out and keeping track of signs when permuting the basis vectors, we rewrite the last sum (omitting the combinatorial factor) as
\begin{align*}
&\sum_{i,j_{1},\dots,j_{k-1}}\left<(\nabla_{\eps_{j_{1}}}\psi,\eps_{jij_{2}\dots j_{k-1}}),(\psi,\eps_{j_{1}ij_{2}\dots j_{k-1}})\right>+\dots\\
&\ \ +\left<(\nabla_{\eps_{j_{q}}}\psi,\eps_{jj_{1}\dots j_{q-1}ij_{q+1}\dots j_{k-1}}),(\psi,\eps_{j_{q}j_{1}\dots j_{q-1}ij_{q+1}\dots j_{k-1}})  \right>+\dots\\
&\ \ +\left<(\nabla_{\eps_{j_{k-1}}}\psi,\eps_{jj_{1}\dots j_{k-2} i}),(\psi,\eps_{j_{k-1}j_{1}\dots j_{k-2}i})\right>\\
&=(k-1)\sum_{i,j_{1},\dots,j_{k-1}}\left<(\nabla_{\eps_{i}}\psi,\eps_{jj_{1}\dots j_{k-1}}),(\psi,\eps_{ij_{1}\dots j_{k-1}})\right>\\
&=(k-1)\cdot(k-1)!\sum_{i}\sum_{J}\left<(\nabla_{\eps_{i}}\psi,\eps_{j}\wedge\eps_{J}),(\psi,\eps_{i}\wedge\eps_{J}) \right>,
\end{align*}
where the indices were relabeled in the second last line. Thus (\ref{simpemdiv}) reduces to
\begin{align*}
&\sum_{i}\sum_{J^{k-1}}\left[\left<(d^{\nabla}\psi,\eps_{j}\wedge\eps_{i}\wedge\eps_{J}),(\psi,\eps_{i}\wedge\eps_{J})\right> + k\left<(\nabla_{\eps_{i}}\psi,\eps_{j}\wedge\eps_{J}),(\psi,\eps_{i}\wedge\eps_{J})\right>
   \right]\\
&=k\left\{\sum_{L^{k}}\left<(d^{\nabla}\psi,\eps_{j}\wedge\eps_{L}),(\psi,\eps_{L})\right> + \sum_{i}\sum_{J^{k-1}}\left<(\nabla_{\eps_{i}}\psi,\eps_{j}\wedge\eps_{J}),(\psi,\eps_{i}\wedge\eps_{J})\right>\right\}.
\end{align*}
The result follows, since the latter term cancels out the unwanted term in the expression for $\mbox{div}\ T_{\psi}^{g}$ above.
\end{proof}
We therefore see that the following conservation law for $p$-harmonic $k$-forms may be read off this formula.
\begin{cor}[Conservation Law]\label{conslaw}If $\psi$ is $p$-harmonic, then $\mbox{div}\ T_{\psi}^{g}=0$.\end{cor}
In a sense, the energy-momentum tensor is thought to contain information about how $p$-harmonic vector bundle-valued $k$-forms scale. In \cite{alikakos2011some}, for example, the integral of the divergence of $T_{\psi}^{g}$ contracted with the radial vector field $x\mapsto \sum_{i}\frac{x^{i}}{|x|}\left.\partial_{i}\right|_{x}\in T_{x}\mathbb{R}^{n}$ yields an expression that co\"incides with what is usually obtained after scaling the integrand of the localised average Dirichlet energy associated to the equation and differentiating. In order to make use of this technique more generally, we compute the divergence of the energy-momentum tensor contracted with an arbitrary (local) vector field, henceforth to be interpreted as a `scaling direction.' The following proposition, a general product rule formula, shall be made use of in the sequel.
\begin{prop}\label{contem} If $U\subset M$ is open, $X\in\Gamma(TU)$, and $S\in \Gamma(T^{\ast}M\otimes T^{\ast}M)$ is symmetric, the identity
\begin{align*}
\Div\ \iota_{X}S&=\left<S,\nabla X^{\flat}\right> + \iota_{X}\Div\ S
\end{align*}
where $(\cdot)^\flat:TM\rightarrow T^{\ast}M$ is the `musical isomorphism' induced by $g$ and $\iota_{X}$ denotes the interior product associated to $X$ acting on $\Gamma(\bigotimes T^{\ast}U)$ by `fixing the first entry.'
\end{prop}
\begin{proof}
Write $(\cdot)^{\sharp}={(\cdot)^{\flat}}^{-1}$. We again compute in a local orthonormal frame adapted at $x$ so that, at $x$,
\begin{align*}
\mbox{div}\ \iota_{X}S&=\sum_{i=1}^{n}\left<\nabla_{\eps_{i}}\left(\iota_{X}S\right)^{\sharp},\eps_{i}\right>\\
&=\sum_{i=1}^{n}\left<\nabla_{\eps_{i}}\left(\iota_{X}S\right),\omega^{i}\right>\\
&=\sum_{i=1}^{n}\left<\iota_{\nabla_{\eps_{i}}X}S + \iota_{X}\nabla_{\eps_{i}}S,\omega^{i}\right>\\
&=\left<S,\nabla X\right> + \iota_{X}\mbox{div}\ S,
\end{align*}
where we used the symmetry of $S$ in the last step.
\end{proof}

\chapter{Monotonicity formul\ae}\label{monsec}
We now make use of the identities of the preceding section to derive monotonicity formul\ae\ for $p$-harmonic $k$-forms. To this end, fix $x_{0}\in M$, write $i_{0}$ for the injectivity radius at $x_{0}$ and $r=d(x_{0},\cdot)$ for the \textit{distance} function measured from $x_{0}$. Decomposing the metric $g$ as
\begin{align*}
g=g_{r}+\D r\otimes\D r
\end{align*}
in $B_{i_{0}}(x_{0})$, we note the local geometry estimate (in the sense of bilinear forms)
\begin{align}\label{locgeomest}
\underline{\Lambda}r^{2}g_{r}\leq g-\nabla^{2}\left(\frac{1}{2}r^{2}\right)\leq \overline{\Lambda} r^{2}g_{r}
\end{align}
on $B_{\frac{i_{0}}{2}}(x_{0})$, where $\nabla^{2}$ is the \textit{Hessian} operator and $\underline{\Lambda},\overline{\Lambda}\in\mathbb{R}^{+}$ are constants depending on the geometry of $M$ in $B_{i_{0}}(x_{0})$ (cf. e.g. \cite[Theorem 27]{petersen}). More specifically, if $M=\mathbb{H}^{n}_{-\kappa^{2}}$, i.e. the upper-half space $\mathbb{R}^{n}_{+}$ equipped with a metric of nonpositive constant sectional curvature $-\kappa^{2}$ ($\kappa\geq 0$), then $i_{0}=\infty$ and the equality
\begin{align*}
g-\nabla^{2}\left(\frac{1}{2}r^{2}\right)=\begin{cases}\left(1-\kappa r\coth(\kappa r)\right)g_{r},&\kappa>0\\0,&\kappa=0  \end{cases}
\end{align*}
holds on all of $M$.
\begin{thm}[Monotonicity formula]\label{monformtheorem} Let $\psi\in \Gamma(E\otimes\Lambda^{k}T^{\ast}M)$. The identity
	\begin{multline}
		\frac{\D}{\D R}\left(\frac{1}{R^{n-kp}}\int_{B_{R}(x_{0})}e_{g}(\psi)\dvol_{g} \right)\\
		= \frac{1}{R^{n-kp+1}}\int_{B_{R}(x_{0})}\left<T_{\psi}^{g},g-\nabla^{2}\left(\frac{1}{2}r^{2}\right)\right> - \iota_{r\nabla r}\Div\ T_{\psi}^{g}\dvol_{g}\\
		+\frac{1}{R^{n-kp}}\int_{\partial B_{R}(x_{0})}|\psi|^{p-2}|\iota_{\nabla r}\psi|^{2}\D S\label{monident}
	\end{multline}
	holds for all $R\in\left]0,i_{0}\right[$. In particular, if $\psi$ is $p$-harmonic, then there exists a constant $\Lambda\geq 0$ depending on the geometry of $M$ in $B_{i_{0}}(x_{0})$ such that
	\begin{align}
		\frac{\D}{\D R}\left(\frac{e^{\Lambda R^{2}}}{R^{n-kp}}\int_{B_{R}(x_{0})}e_{g}(\psi)\dvol_{g} \right)\geq 0\label{monformgen}
	\end{align}
	for $R< \frac{i_{0}}{2}$. In particular, if $M=\mathbb{H}^{n}_{-\kappa^{2}}$, then this inequality holds for all $R>0$ with $\Lambda=0$.
\end{thm}
\begin{proof} Taking $X=\nabla\left(\frac{1}{2}r^{2}\right)=r\nabla r$ and $Y=\iota_{X}T_{\psi}^{g}$, it is clear from Proposition \ref{contem} that
	\begin{align}
		\Div\ Y&=\tr\ T_{\psi}^{g} - \left<T_{\psi}^{g},g-\nabla^{2}\left(\frac{1}{2}r^{2}\right)\right> + \iota_{r\nabla r}\Div\ T_{\psi}^{g}\\
		&=(kp-n)e_{g}(\psi)- \left<T_{\psi}^{g},g-\nabla^{2}\left(\frac{1}{2}r^{2}\right)\right> + \iota_{r\nabla r}\Div\ T_{\psi}^{g}
	\end{align}
	whence, by Gau\ss' theorem,
	\begin{align*}
		\int_{\partial B_{R}(x_{0})}\left<Y,\nabla r\right>\D S&=R\int_{\partial B_{R}(x_{0})}|\psi|^{p-2}|\iota_{\nabla r}\psi|^{2}-e_{g}(\psi)\D S\\
		&=(kp-n)\int_{B_{R}(x_{0})}e_{g}(\psi)\dvol_{g}\\
		&\qquad - \int_{B_{R}(x_{0})}\left<T_{\psi}^{g},g-\nabla^{2}\left(\frac{1}{2}r^{2}\right)\right>-\iota_{r\nabla r}\Div\ T_{\psi}^{g}\dvol_{g}
	\end{align*}
	which may be rearranged as
	\begin{multline}
		(kp-n)\int_{B_{R}(x_{0})}e_{g}(\psi)\dvol_{g} + R\int_{B_{R}(x_{0})}e_{g}(\psi)\D S\\
		=R\int_{\partial B_{R}(x_{0})}|\psi|^{p-2}|\iota_{\nabla r}\psi|^{2}\D S + \int_{B_{R}(x_{0})}\left<T_{\psi}^{g},g-\nabla^{2}\left(\frac{1}{2}r^{2}\right)\right>-\iota_{r\nabla r}\Div\ T_{\psi}^{g}\dvol_{g}.\label{almostmon}
	\end{multline}
	By the coarea formula,
	\begin{align*}
		R^{n-kp+1}\frac{\D}{\D R}\left(\frac{1}{R^{n-kp}}\int_{B_{R}(x_{0})}e_{g}(\psi)\dvol_{g}\right)=(kp-n)\int_{B_{R}(x_{0})}e_{g}(\psi)\dvol_{g} + R\int_{B_{R}(x_{0})}e_{g}(\psi)\D S
	\end{align*}
	which, together with (\ref{almostmon}), implies (\ref{monident}).

	Now suppose that $\psi$ is $p$-harmonic so that $\Div\ T_{\psi}^{g}=0$ by Corollary \ref{conslaw}. Assuming the local geometry estimate (\ref{locgeomest}), it is clear that
	\begin{align*}
		\left<T_{\psi}^{g},g-\nabla^{2}\left(\frac{1}{2}r^{2}\right)\right>\geq \left(kp\underline{\Lambda}^{-}- (n-1)\overline{\Lambda}\right)_{-}R^{2}e_{g}(\psi) - \underline{\Lambda}^{-}|\psi|^{p-2}|\iota_{\nabla r}\psi|^{2},
	\end{align*}
	on $B_{R}(x_{0})$, where for $\alpha\in\mathbb{R}$ we write $\alpha^{-}=\min\{\alpha,0\}$. Thus, setting $\Lambda=-\frac{1}{2}\left(kp\underline{\Lambda}^{-}- (n-1)\overline{\Lambda}\right)_{-}$, (\ref{monident}) implies that
	\begin{multline*}
		\frac{\D}{\D R}\left(\frac{1}{R^{n-kp}}\int_{B_{R}(x_{0})}e_{g}(\psi)\dvol_{g}\right)+\frac{2\Lambda R}{R^{n-kp}}\int_{B_{R}(x_{0})}e_{g}(\psi)\dvol_{g}\\
		\geq\frac{1}{R^{n-kp}}\int_{\partial B_{R}(x_{0})}|\psi|^{p-2}|\iota_{\nabla r}\psi|^{2}\D S - \frac{\underline{\Lambda}^{-}}{R^{n-kp+1}}\int_{B_{R}(x_{0})}|\psi|^{p-2}|\iota_{\nabla r}\psi|^{2}\dvol_{g}\geq 0.
	\end{multline*}
	Multiplying through by the integrating factor $e^{\Lambda R^{2}}$ then implies (\ref{monformgen}).

	Finally, for the case where $M=\mathbb{H}^{n}_{-\kappa^{2}}$, we note that the case $\kappa=0$ follows from the preceding computations with $\underline{\Lambda}=\overline{\Lambda}=0$, whereas the case $\kappa>0$ follows from an explicit computation, namely by noting that
	\begin{align*}
		\left<T_{\psi}^{g},g-\nabla^{2}\left(\frac{1}{2}r^{2}\right)\right>=(1-\kappa r\coth(\kappa r))\left((kp-(n-1))e_{g}(\psi) - |\psi|^{p-2}|\iota_{\nabla r}\psi|^{2}\right),
	\end{align*}
	but $1-\kappa r\coth(\kappa r)\leq 0$ for all $r>0$ and $kp-(n-1)\leq 0$ so that this expression is nonnegative, whence the final claim follows from (\ref{monident}).
\end{proof}
Besides implying a monotonicity formula for $p$-harmonic bundle-valued $k$-forms, Propositions \ref{emdiv} and \ref{contem} also yields a monotonicity formula for bundle-valued $k$-forms with suitably controlled `inhomogeneities' as indicated in the following theorem which should be compared to \cite[\S4.3]{simongmt} and \cite[Theorem 3.2]{uhlenbeckpriori}.
\begin{thm} If $\psi\in\Gamma(E\otimes\Lambda^{k}T^{\ast}M)$ is such that
	\begin{align*}
 |\delta^{\nabla}(|\psi|^{p-2}\psi)| + |\psi|^{p-2}|\D^{\nabla}\psi|\leq \Gamma
	\end{align*}
	on $B_{i_{0}}(x_{0})$, then
	\begin{align*}
		\frac{\D}{\D R}\left(\frac{e^{\Lambda R^{2} + R}}{R^{n-kp}}\int_{B_{R}(x_{0})}e_{g}(\psi)\dvol_{g}+
	      \frac{{\Gamma}^{p'}}{p'}\int_{0}^{R}\frac{e^{\Lambda s^{2} + s}}{s^{n-kp}}\Vol(B_{s}(x_{0}))\D s\right)\geq 0
	\end{align*}
	for $R<\frac{i_{0}}{2}$, where $\Lambda\in\mathbb{R}$ is as in Theorem \ref{monformtheorem}, $\Vol(B_{s}(x_{0}))=\int_{B_{s}(x_{0})}\dvol_{g}$ and $p'>1$ is such that $\frac{1}{p}+\frac{1}{p'}=1$.
\end{thm}
\begin{proof}
Set $q_{\psi}=|\delta^{\nabla}(|\psi|^{p-2}\psi)| + |\psi|^{p-2}|\D^{\nabla}\psi|$. Using the Cauchy-Schwarz inequality, it is clear that
\begin{align*}
	-\iota_{r\nabla r}\Div\ T_{\psi}^{g}&=\left<\delta^{\nabla}(|\psi|^{p-2}\psi),\iota_{r\partial_{r}}\psi\right> + |\psi|^{p-2}\left<\iota_{r\partial_{r}}\D^{\nabla}\psi,\psi\right>\\
&\geq -R\left\{|\delta^{\nabla}(|\psi|^{p-2}\psi)|\cdot|\iota_{\partial_{r}}\psi| + |\psi|^{p-2}|\iota_{\partial_{r}}\D^{\nabla}\psi|\cdot|\psi|\right\}\\
&\geq -R\ |\psi|q_{\psi},
\end{align*}
whenever $r<R$, whence an application of Young's inequality yields
\begin{align*}
	-\iota_{r\nabla r}\Div\ T_{\psi}^{g}&\geq -R\left(\frac{|\psi|^{p}}{p}+\frac{q_{\psi}^{p'}}{p'}\right)\\
&\geq -R\left(e_{g}(\psi) + \frac{\Gamma^{p'}}{p'}\right).
\end{align*}
Therefore, this together with the identity (\ref{monident}) and geometry bounds (\ref{locgeomest}) implies that
\begin{align*}
	\frac{\D}{\D R}\left(\frac{1}{R^{n-kp}}\int_{B_{R}(x_{0})}e_{g}(\psi)\dvol_{g} \right) + \frac{2\Lambda R}{R^{n-kp}}\int_{B_{R}(x_{0})}e_{g}(\psi)\dvol_{g} + \frac{\Gamma^{p'}}{p'}\cdot\frac{\Vol(B_{R}(x_{0}))}{R^{n-kp}}\geq 0.
\end{align*}
Thus, noting that $\Vol(B_{R}(x_{0}))=O(R^{n})$ as $R\searrow 0$ and multiplying through by the integrating factor $e^{\Lambda R^{2} + R}$ then establish the claim.
\end{proof}
This theorem applies e.g. if one replaces the $p$-coclosed condition on $\psi$ in Examples \ref{hmex} and \ref{ymex} with an equation of the form $\delta^{\nabla}(|\psi|^{p-2}\psi)=Q$, where $Q$ is some bounded form (cf. \cite{uhlenbeckpriori}).
\chapter{The Yang-Mills-Higgs system}\label{ymhsec}
We now turn our attention to a system not cast in the form (\ref{closed})-(\ref{coclosed}), but which is in some sense a coupling of Example \ref{ymex} and the semilinear elliptic system (\ref{semilineqn})--- the Yang-Mills-Higgs system (cf. e.g. \cite{MR614447})--- to which the techniques developed here readily lend themselves. To this end, assume the setup of Example \ref{ymex} and let $\rho:G\rightarrow\textup{GL}(V)$ be a representation of $G$ on the $\mathbb{R}$-vector space $V$ which is assumed to be equipped with a $\rho$-invariant inner product $\left<\cdot,\cdot\right>$. Together with the data of $P$, $(V,\left<\cdot,\cdot\right>)$ and $\rho$ give rise to a vector bundle $E_{0}$ admitting a connection $\nabla^{0}$ (induced by a connection $\omega$ on $P$) compatible with $\left<\cdot,\cdot\right>$. A pair $(\omega,u)$ consisting of a connection $\omega$ on $P$ and a section $u\in\Gamma(E_{0})$ is then said to be a \textit{Yang-Mills-Higgs pair} with (symmetric) potential $W\in C^{\infty}(\mathbb{R},\left[0,\infty\right[)$ whenever the equations
\begin{align}
	\delta^{\nabla}\underline{\Omega}^{\omega} + u\odot \D^{\nabla^{0}}u &=0\nonumber\\
	\delta^{\nabla^{0}}\D^{\nabla^{0}}u + 2(W'\circ |u|^{2})u&=0\label{ymhe}\tag{YMHE}
\end{align}
are satisfied, where $\odot:E_{0}\times E_{0}\otimes \Lambda T^{\ast}M\rightarrow E\otimes \Lambda T^{\ast}M$ is a fibrewise bilinear map defined such that
\begin{align*}
	\left<X,e_{1}\odot e_{2}\right>=\left<X\cdot e_{1},e_{2}\right>,
\end{align*}
for all $X\in \Gamma(E)$ and $e_{1},e_{2}\in \Gamma(E_{0})$, where $\cdot$ is the natural action of $E$ on $E_{0}$ induced by the derivative of $\rho$ at the identity of $G$, and extended to the rest of $E_{0}\times E_{0}\otimes \Lambda T^{\ast}M$ such that for all $\eta\in \Gamma(\Lambda T^{\ast}M)$ and $e_{1},e_{2}\in\Gamma(E_{0})$, $e_{1}\odot(e_{2}\otimes\eta)=(e_{1}\odot e_{2})\otimes \eta$. These equations arise from the variational problem
\begin{align*}
	\int_{M}\frac{1}{2}|\underline{\Omega}^{\xi}|^{2} + \frac{1}{2}|\D^{\nabla^{0}}v|^{2} + W\circ |v|^{2} \dvol_{g}\rightarrow\textup{min!}
\end{align*}
considered over the class of all connections $\xi$ on $P$ and sections $v\in\Gamma(E_{0})$ on which this integral is finite.

Now, suppose $\omega$ is an arbitrary connection on $P$ and $u$ an arbitrary section of $E_{0}$. In this case, the energy density to consider is
\begin{align*}
	\widetilde{e}_{g}(\omega,u)=\frac{1}{2}|\underline{\Omega}^{\omega}|^{2} + \frac{1}{2}|\D^{\nabla^{0}}u|^{2} + W\circ |u|^{2}.
\end{align*}
Computing exactly as in Propositions \ref{emdef} and \ref{emdiv}, we are led to the energy-momentum tensor
\begin{align*}
	\widetilde{T}_{(\omega,u)}^{g}=\sum_{i,j=1}^{n}\left(\left<\iota_{\eps_{i}}\underline{\Omega}^{\omega},\iota_{\eps_{j}}\underline{\Omega}^{\omega}\right> + \left<\nabla_{i}^{0}u,\nabla_{j}^{0}u\right>\right)\omega^{i}\otimes \omega^{j} - \widetilde{e}_{g}(\omega,u)g,
\end{align*}
expressed in any local frame $\{\eps_{i}\}\leftrightarrow\{\omega^{i}\}$, and the expression
	\begin{align*}
		\Div\ \widetilde{T}_{(\omega,u)}^{g}=-\sum_{j=1}^{n}\left(\left< \delta^{\nabla}\underline{\Omega}^{\omega} + u\odot \D^{\nabla^{0}}u,\iota_{\eps_{i}}\underline{\Omega}^{\omega} \right> + \left< \delta^{\nabla^{0}}\D^{\nabla^{0}}u + 2(W'\circ |u|^{2})u,\nabla^{0}_{i}u \right> \right)\omega^{j}
	\end{align*}
	for its divergence, thus implying a conservation law in this case if $(\omega,u)$ is a Yang-Mills-Higgs pair. Proceeding exactly as in Theorem \ref{monformtheorem}, noting in particular that
       \begin{align*}
	       \tr\ \widetilde{T}_{(\omega,u)}^{g}=(4-n)\widetilde{e}_{g}(\omega,u)-\left(|\D^{\nabla}u|^{2} + 4W\circ |u|^{2}\right),
       \end{align*}
       then yields the following monotonicity principle.
\begin{thm}\label{monformymhthm} Suppose $\omega$ is a connection on $P$ and $u\in\Gamma(E_{0})$. The identity
	\begin{multline*}
		\frac{\D}{\D R}\left(\frac{1}{R^{n-4}}\int_{B_{R}(x_{0})}\widetilde{e}_{g}(\omega,u)\dvol_{g} \right)\\
		=\frac{1}{R^{n-4}}\int_{\partial B_{R}(x_{0})}|\psi|^{p-2}|\iota_{\nabla r}\psi|^{2}\D S + \frac{1}{R^{n-3}}\int_{B_{R}(x_{0})}|\D^{\nabla^{0}}u|^{2} + 4W\circ|u|^{2} \dvol_{g}\\
		+ \frac{1}{R^{n-3}}\int_{B_{R}(x_{0})}\left<T_{\psi}^{g},g-\nabla^{2}\left(\frac{1}{2}r^{2}\right)\right> - \iota_{r\nabla r}\Div\ T_{\psi}^{g}\dvol_{g}
	\end{multline*}
	holds for all $R\in\left]0,i_{0}\right[$. In particular, if $(\omega,u)$ is a Yang-Mills-Higgs pair, then there exists a constant $\Lambda\geq 0$ depending on the geometry of $M$ in $B_{i_{0}}(x_{0})$ such that
	\begin{align}
		\frac{\D}{\D R}\left(\frac{e^{\Lambda R^{2}}}{R^{n-4}}\int_{B_{R}(x_{0})}\widetilde{e}_{g}(\omega,u)\dvol_{g} \right)\geq 0\label{monformgenymh}
	\end{align}
	for $R< \frac{i_{0}}{2}$. In particular, if $M=\mathbb{H}^{n}_{-\kappa^{2}}$, then this inequality holds for all $R>0$ with $\Lambda=0$.
\end{thm}
\chapter{Application: Liouville-type theorems}\label{liouvillesec}
The following are immediate applications of Theorems \ref{monformtheorem} and \ref{monformymhthm}.
\begin{thm} Suppose $M=\mathbb{H}^{n}_{-\kappa^{2}}$ and $\psi\in\Gamma(E\otimes\Lambda^{k}T^{\ast}M)$ is $p$-harmonic. If
	\begin{equation*}
		\int_{B_{R}(x_{0})}e_{g}(\psi)\dvol_{g}=o(R^{n-kp})
	\end{equation*}
as $R\rightarrow\infty$, then $\psi\equiv 0$.
\end{thm}
\begin{proof}
By Theorem \ref{monformtheorem}, it is clear from (\ref{monformgen}) that for all $0<R_{0}<R$ and $x_{0}\in M$,
\begin{align*}
	\frac{1}{R_{0}^{n-kp}}\int_{B_{R_{0}}(x_{0})}e_{g}(\psi)\dvol_{g}\leq \frac{1}{R^{n-kp}}\int_{B_{R}(x_{0})}e_{g}(\psi)\dvol_{g}\xrightarrow{R\rightarrow\infty}0
\end{align*}
so that $e_{g}(\psi)\equiv 0$.
\end{proof}
\begin{thm} Suppose $M=\mathbb{H}^{n}_{-\kappa^{2}}$ and $(\omega,u)$ is a Yang-Mills-Higgs pair. If
	\begin{equation*}
		\int_{B_{R}(x_{0})}\widetilde{e}_{g}(\omega,u)\dvol_{g}=o(R^{n-4})
	\end{equation*}
as $R\rightarrow\infty$, then $\omega$ is flat and $u$ is parallel with respect to $\nabla^{0}$. Thus, for an appropriate global section $\sigma:M\rightarrow P$, $\sigma^{\ast}\omega=0$ and $u$ may be represented as a constant function on $M$ relative to the trivialization of $E_{0}$ induced by $\sigma$.
\end{thm}
\begin{proof}
	The first claim follows from $\widetilde{e}_{g}(\omega,u)=0$, which is established exactly as in the preceding theorem, and the latter from \cite[Corollary 9.2]{MR0152974}, since the simply-connectedness of $\mathbb{R}^{n}_{+}$ and $\underline{\Omega}^{\omega}\equiv 0$ imply that there exists a global section $\sigma:M\rightarrow P$ of $P$ such that $\sigma^{\ast}\omega=0$, whence $\D^{\nabla}$ reduces to the usual exterior derivative (acting on vector-valued differential forms) when considered in the trivialisation of $E_{0}$ induced by $\sigma$ so that $u$ may be represented by a constant function relative to this trivialisation.
\end{proof}

\end{document}